\documentclass{article}

\usepackage{graphicx,color}
\usepackage{pict2e}
\usepackage{amsthm,amsmath}
\usepackage{amssymb}
\usepackage{mathrsfs}

\newtheorem{theorem}{Theorem}[section]
\newtheorem{coro}[theorem]{Corollary}
\newtheorem{lemma}[theorem]{Lemma}
\newtheorem{conj}[theorem]{Conjecture}

\title{Mader's conjecture for spiders}
\author{Yanmei Hong, Qinghai Liu}
\date{\today}

\begin{document}

\title{Mader's conjecture for graphs with small connectivity\thanks{Research supported by NSFC(11871015, 11401103) and FJSFC (2018J01665). }}
\author{Yanmei Hong\thanks{College of Mathematics and Computer
Science, Fuzhou University, Fuzhou, 350108, China. Email: yhong@fzu.edu.cn}\thanks{Corresponding author.} \quad Qinghai Liu  \thanks{Center for Discrete Mathematics, Fuzhou
University, Fuzhou, Fujian, 350002, Email: qliu@fzu.edu.cn }}

\date{}

\maketitle

\begin{abstract}
Mader conjectured that for any tree $T$ of order $m$, every $k$-connected graph $G$ with minimum degree at least $\lfloor\frac{3k}{2}\rfloor +m-1$ contains a subtree $T'\cong T$ such that $G-V(T')$ is $k$-connected. In this paper, we give a characterization for a subgraph to contain an embedding of a specified tree avoiding some vertex. As a corollary, we confirm Mader's conjecture for $k\leq3$.
\end{abstract}



\section{Induction}
The graphs considered in this paper are finite, undirected, and simple (no loops or multiple edges).
For a graph $G$, we use $V(G)$, $E(G)$, $\delta(G)$ and $\kappa(G)$ to denote its \textit{vertex set}, \textit{edge set}, \textit{minimum degree} and \textit{connectivity}, respectively.  The \textit{order} of a graph $G$ is the cardinality of its vertex set, denoted by $|G|$.  By $H\subseteq G$ we mean that $H$ is a subgraph of $G$, and we view any subset of vertices as a subgraph with no edges. For any $U\subseteq G$, the \textit{neighborhood} of $U$, denoted by $N_G(U)$, is the set of vertices in $V(G)-U$ adjacent to at least one vertex in $U$. If the graph $G$ is clear from the context, the reference to $G$ is usually omitted.
 For graph-theoretical terminologies and notation not defined here, we follow\cite{bondy}.  

The following well-known result due to Chartrand, Kaugars and Lick \cite{ckl}.
\begin{theorem}[Chartrand, \textit{et. al.}\cite{ckl}]
Every $k$-connected graph $G$ with  $\delta(G)\geq \lfloor  \frac {3k}{2} \rfloor $  has a vertex  $x $ such that $G-x$ is still $k$-connected.
\end{theorem}


Fujita and Kawarabayashi \cite{kk} studied the question that whether one can choose the minimum degree
$\delta(G)$ of a $k$-connected graph $G$ large enough (independent of the number of vertices $|G|$ of $G$) to contain adjacent vertices $u\neq v$ such that $G-\{u,v\}$ is still $k$-connected, and gave a positive answer. Also, in the same paper, they conjectured that one can find a function $f_k(m)$ such that every $k$-connected graph $G$ with $\delta(G)\geq \lfloor\frac{3k}{2}\rfloor+f_k(m)-1$ contains a connected subgraph $W$ of order $m$ such that $G-V(W)$ is still $k$-connected. They also gave examples in \cite{kk} to show that $f_k(m)\geq m$  for all positive integers $k$, $m$. In \cite{mader1}, Mader confirmed the conjecture by showing that $f_k(m)=m$. In fact, Mader showed the connected subgraph $W$ can be chosen to be a path.
Based on the result, Mader conjectured that the $W$ can be chosen to be every tree of order $m$ and
made the following conjecture. A \textit{graph isomorphism} from a graph $G$ to a graph $H$, written $\phi:G\to H$, is a mapping $\phi: V(G)\to V(H)$ such that $uv\in E(G)$ if and only if $\phi(u)\phi(v)\in E(H)$.

  \begin{conj}[Mader \cite{mader1}]\label{conj-mader}
 For any tree $T$ of order $m$, every $k$-connected graph $G$ with $\delta(G)\geq \lfloor \frac{3k}{2}\rfloor +m-1$ contains a subtree $T'\cong T$ such that $G-V(T')$ is $k$-connected.
  \end{conj}

  Concerning to this conjecture, Mader also proved that a lower bound on the minimum degree which is independent of the order of $G$ indeed exists for any $m$ and $k$.

  \begin{theorem}[Mader \cite{mader2}]
  For any tree $T$ of order $m$, every $k$-connected graph $G$ with $\delta(G)\geq 2(k-1+m)^2+m-1$ contains a subtree $T'\cong T$ such that $G-V(T')$ is $k$-connected.
  \end{theorem}


  Until now, Mader's conjecture has been shown for $k=1$ by Diwan and Tholiya \cite{dt}. For $k=2$, only some classes of trees have been verified. In \cite{tian1, tian2},  stars, double-stars, path-stars, and path-double-stars are verified due to Tian, Lai, Xu and Meng. In \cite{lp1}, trees with diameter at most 4 are verified due to Lu and Zhang. In \cite{ho}, trees with small internal vertices and quasi-monotone caterpillars are verified due to Hssunuma and Ono.
  Very recently, caterpillars and spiders are verified due to Hong, Liu, Lu and Ye \cite{hlly}.
  All the methods are concerned with the structure of the trees.
  In this paper, we develop a general way that works for arbitrary trees and confirm Mader's conjecture for all classes of trees when $k=2$ or 3. Our method might have some potential generalization into $k$-connected case.

  We conclude this section with some notations about trees.
 A rooted tree is a tree with a specific vertex, which is call \textit{root}.
  An \textit{$r$-tree} is a rooted tree with root $r$. For an $r$-tree $T$ and a vertex $v\in V(T)$, $rTv$ is defined as the unique path from $r$ to $v$ in $T$. Every vertex in the path $rTv$ is called an \textit{ancestor} of $v$, and each vertex of which $v$ is an ancestor is a \textit{descendant} of $v$. An ancestor or descendant of a vertex is \textit{proper} if it is not the vertex itself. The immediate proper ancestor of a vertex $v$ other than the root is its \textit{parent} and the vertices whose parent is $v$ are its \textit{children}. A \textit{leaf} of a rooted tree is a vertex with no children. The set of leaves of an $r$-tree $T$ is denoted by $Leaf(T)$. For a vertex $v$ in an $r$-tree $T$, the \textit{subtree} of $T$, denoted by $T_v$, is the tree induced by all its descendant and with root $v$.


\section{Preliminaries}

By greedy strategy, it is easy to see that every graph with minimum degree at least $m-1$ contains every tree of order $m$. In fact, one may obtain a slightly general form as following two lemmas.

\begin{lemma}[Hasunuma and Ono \cite{ho}]\label{lem-tree}
  Let $T$ be a tree of order $m$ and $T_1$ a subtree of $T$. If a graph $G$ contains a subtree $T_1'\cong T_1$, denoting by $\phi: T_1\to T_1'$ the isomorphism, such that $d_G(v)\geq m-1$ for any $v\in V(G)\setminus \{\phi(x)\mid x\in V(T_1), N_T(x)\subseteq V(T_1)\})$, then $G$ contains a subtree $T'\cong T$ such that $T_1'\subseteq T'$.
\end{lemma}

\begin{lemma}[Hasunuma and Ono \cite{ho}]\label{lem-tree2}
    Let $T$ be a tree of order $m$, $U\subseteq V(T)$ and $S$ be a subtree obtained from $T$ by deleting some leaves adjacent to a vertex in $U$. If a graph $G$ contains a subtree $S' \cong S$ such that $d_G(u) \ge m-1$ for any $u\in U' = \{\phi(x):x \in U\}$ for an isomorphism $\phi$ from $S$ to $S'$, then $G$ contains a subtree $T' \cong T$ such that $S' \subseteq T'$.
\end{lemma}

By merging the above two lemmas, we may get the following more general form which will be used to construct trees in our proof.

\begin{lemma}\label{lem-tree3}
 Let $T_0$ be a tree of order $m$ and $T'$ a subtree of $T_0$.
 Let $(L_1,\ldots, L_{k+1})$ be a partition of $V(T_0)\setminus V(T')$ such that for $i=0,\ldots, k-1$, $L_{i+1}\subseteq Leaf(T_i)$ and $T_{i+1}=T_i-L_{i+1}$ (note that vertices in $L_{k+1}$ may not be leaves).
Assume $G$ is a graph contains a subgraph $H'\cong T'$,  denoting by $\phi: T'\to H'$ the isomorphism, and $V(G)\setminus V(H')$ has a partition $(X_1,\ldots,X_{k+1})$ such that each of the followings hold, cohere $X_{k+2}=\{\phi(v)\mid v\in V(T'), N_{T_0}(v)\not\subseteq V(T')\}$.

\indent(a) $|N_G(x)\setminus \bigcup_{j=1}^k X_{j}|\geq m-1-\sum_{j=1}^{k}|L_j|$ for each $x\in X_{k+1}\cup X_{k+2}$;\\
\indent(b) $|N_G(x)\setminus \bigcup_{j=1}^{i-2} X_j|\geq m-1-\sum_{j=1}^{i-2}|L_j|$ for each $i=2,\ldots, k+1$ and for each $x\in \bigcup_{j=i}^{k+2} X_j$.

\noindent Then $G$ contains a subtree $H\cong T$ such that $H'\subseteq H$.
\end{lemma}

\begin{proof}
  The proof is by induction on $k$. If $k=0$ then the condition (a) is exactly the degree condition in Lemma \ref{lem-tree} and the condition (b) is nothing. Thus the result follows Lemma \ref{lem-tree}. So we may assume $k\geq1$ and the result holds for any smaller $k$. Let $G'=G-X_1$. Then $T'$ is a subtree of $T_1$, $H'$ is a subgraph of $G'$, $(L_2,\ldots, L_{k+1})$ is a partition of $V(T_1)\setminus V(T')$ and $(X_2,\ldots, X_{k+1})$ is a partition of $G'-V(H')$. It is easy to see that the conditions holds for $T_1$ and $G'$. Then by induction hypothesis, $G'$ contains a subgraph $H_1\cong T_1$ such that $H'\subseteq H_1$. Denote by $\phi':T_1\to H_1$ the isomorphism.

  Let $L_{k+2}=\{v\mid v\in V(T'), N_T(v)\not\subseteq V(T')\}$. Then $\phi'(L_{k+2})=\phi(L_{k+2})=X_{k+2}$.
  By (b) with $i=2$, we see that $d_G(x)\geq m-1$ for any $x\in \bigcup_{i=2}^{k+2} X_i=(V(G')\setminus V(H'))\cup X_{k+2})$. By the assumption of $L_{k+2}$, $N_T(L_1)\subseteq \bigcup_{i=2}^{k+2} L_i$. 
  Then by Lemma \ref{lem-tree2}, there is an embedding of $T$ in $G$ and the result holds.
\end{proof}

In our proof, we always apply Lemma \ref{lem-tree3} with $k=2$. Then we may restate the lemma when $k=2$ as follows.

\begin{coro}\label{cor-tree}
   Let $T$ be a tree of order $m$ and $T'$ a subtree of $T$.
 Let $L_1\subseteq Leaf(T)\setminus V(T'), L_2\subseteq Leaf(T-L_1)\setminus V(T')$ and $L_3=V(T)\setminus (V(T')\cup L_1\cup L_2)$.
Assume $G$ is a graph contains a subgraph $H'\cong T'$,  denoting by $\phi: T'\to H'$ the isomorphism, and $V(G)\setminus V(H'))$ has a partition $(X_1,X_2, X_3)$ such that each of the followings hold, cohere $X_{4}=\{\phi(v)\mid v\in V(T'), N_T(v)\not\subseteq V(T')\}$.

\indent(a) $d_G(x)\geq m-1$ for each $x\in X_2\cup X_3\cup X_4$.\\
\indent(b) $|N_G(x)\setminus X_1|\geq m-1-|L_1|$ and $|N_G(x)\setminus (X_1\cup X_2)|\geq m-1-|L_1|-|L_2|$ for each $x\in X_3\cup X_4$.

\noindent Then $G$ contains a subtree $H\cong T$ such that $H'\subseteq H$.
\end{coro}

\section{2-connected graphs}

In this section, we give a characterization for a subgraph to contain an embedding of a specified tree avoiding some vertex. As a corollary, we confirm Mader's conjecture for 2-connected graphs. To apply Corollary \ref{cor-tree}, we need the following counting lemma to satisfy the conditions in Corollary \ref{cor-tree}.

\begin{lemma}\label{lem-leaf}
 Let $T$ be a tree, $T_0\subseteq T$ and $T_1,\ldots, T_k$ be the components of $T-V(T_0)$. If $L_0=Leaf(T_0)$, $|L_1|=Leaf(T)$ and $L_2=Leaf(T-L_1)$ then $|L_1|\geq |L_0|$ and $|L_1|+|L_2|\geq |L_0|+k$.
\end{lemma}
\begin{proof}
For $i=1,\ldots, k$, assume $s_ir_i\in E(T)$ such that $s_i\in V(T_0)$ and $r_i\in V(T_i)$. Note that some $s_i$'s may be the same vertex.
Denote $S=\{s_1,\ldots, s_k\}$ and $R=\{r_1,\ldots, r_k\}$. Then $|S|\leq k=|R|$.

Since each $T_i$ contains a leave of $T$, $$|L_1|\geq k+|L_0\setminus S|=|L_0|+k- |L_0\cap S|\geq |L_0|.$$
Moreover, note that some $T_i$ may contain exactly one vertex and $V(T_i)\cap L_2\neq\emptyset$ if $|V(T_i)|\geq 2$. Let $\bar R=\{r_i\mid |V(T_i)|\geq2\}$. Then $$|L_2|\geq |\bar R|+|L_2\cap S|.$$
For any $s_i\in S\cap L_0$, if $N_T(s_i)\cap \bar R=\emptyset$ then $N_T(s_i)\setminus V(T_0)\subseteq L_1$ and $s_i\in L_2$. It follows that $|N_T(s_i)\cap \bar R|+|\{s_i\}\cap L_2|\geq |\{s_i\}\cap L_0|$. Also, the inequality still holds if $s_i\notin L_0$. By the arbitrariness of $s_i$, we see that $|\bar R|+|S\cap L_2|\geq |S\cap L_0|$. It follows that $|L_1|+|L_2|\geq |L_0|+k - |L_0\cap S| + |\bar R| + |L_2\cap S|\geq |L_0|+k$. The proof is completed.
\end{proof}

By Corollary \ref{cor-tree} and Lemma \ref{lem-leaf}, we have the following lemma which provides a method to spare a vertex in an embedding of a tree.
Applying the lemma several times, we may confirm Mader's conjecture when $k=2$ or $3$.

\begin{lemma}\label{lem-main}
  Let $T_0$ be a tree of order $m$ and $G$ be a graph $\delta(G)\geq \delta\geq m$. Assume $T,B,H$ are disjoint subgraphs of $G$ such that $T\cong T_0$, $V(G)=V(H\cup B\cup T)$.
  Then one of the follows holds.

\indent (a) $N_G(H)\cap V(T)=\emptyset$.\\
\indent (b) $|N_G(v)\cap V(B)|\leq \delta-m$ for each $v\in V(H\cup T)$.\\
\indent (c) There exists $v\in V(H\cup T)$ such that $|N_G(v)\cap V(B)|\geq \delta-m+1$ and $H\cup T-v$ contains a tree $T'\cong T$.
\end{lemma}

\begin{proof}
Assuming none of (a), (b) and (c) holds, we will derive a contradiction. To this end, denote
 $$X=\{v\in V(H\cup T)\mid |N_G(v)\cap V(B)|\geq \delta-m+1\} \mbox{ and }Y=V(T)\cap N_G(H).$$
 Then by the assumptions, we see that  $X, Y\neq\emptyset$. We have the following claim.

  \noindent\textbf{Claim 1.} $X\subseteq V(T)$ and $X\cup Y=V(T)$.

First, $X\subseteq V(T)$ is clear since any vertex in $X\cap V(H)$ satisfies (c). For the second part, suppose that there exists $z\in V(T)\setminus (X\cup Y)$. Then by the definitions of $X, Y$, $|N_G(z)\cap V(B)|\leq \delta-m$ and $N(z)\cap V(H)=\emptyset$. Thus $d_G(z)\leq |N_G(z)\cap V(B)|+|N_G(z)\cap V(T)|\leq \delta-m+m-1=\delta-1$, a contradiction to the assumption of the minimum degree. Claim 1 is proved.

  By Claim 1, pick $u\in Y$ such that $|X\setminus \{u\}|$ is as large as possible. Then $X\setminus \{u\}\neq\emptyset$, for otherwise, $X=\{u\}$ and $Y\setminus\{u\}=\emptyset$, implying $m=1$. Then each vertex in $H$ forms a tree isomorphic to $T_0$ and $u$ is the desired vertex in (c). Consider $T$ as a tree rooted at $u$.

%

  By Claim 1, there exists an edge $w_1w_2\in E(T)$ such that $w_1\in Y$, $w_2\in X$ and $w_1$ is the parent of $w_2$. Subject to this, we may furthermore, assume that $w_1w_2$ is such an edge such that the rooted subtree $T_{w_2}$ is as small as possible. Under the assumption, we see that no edge of $T_{w_2}$ satisfies the conditions of $w_1w_2$, i.e., for any edge $w_1'w_2'$ of $T_{w_2}$ such that $w_1'$ is the parent of $w_2'$, either $w_1'\notin Y$ or $w_2'\notin X$.

  As a corollary, for any $w\in X\cap V(T_{w_2})$, all the ancestors of $w$ in $T_{w_2}$ are in $X$. Thus there is a subtree $T_2\subseteq T_{w_2}$ such that 
  $V(T_{w_2})\cap X=V(T_2)$.
  Let $L_0=Leaf(T_2)$ and $T_3=T_2-L_0$. 
  By the choice of $w_1w_2$,
  \begin{equation}\label{eq-T3H}
  N_G(H)\cap V(T_3)=\emptyset.
  \end{equation}
  Let $R_1,\ldots, R_p$ be the rooted trees obtained from $T_{w_2}$ by deleting the vertices in $T_2$. For $i=1,\ldots, p$, assume $r_i$ is the root of $R_i$ and $s_i$ is the parent of $r_i$ in $T_{w_2}$. 

 \noindent\textbf{Claim 2.} $N_G(T_3)\cap \bigcup_{i=1}^p V(R_i-r_i)\neq\emptyset$.

Suppose, to the contrary, that the claim is not true.
 Let $G'=G-(V(B)\cup \{w_2\})$ and
 pick $h\in N_G(w_1)\cap V(H)$.
Let $T_4=T-V(T_{w_2}-w_2)$ and $T_4'$ is obtained from $T_4$ by replacing $w_2$ with $h$. Then $T_4'\cong T_4\subseteq T$.

Define $L_1=Leaf(T_{w_2})$, $L_2=Leaf(T_{w_2}-L_1)$ and $L_3=V(T_{w_2}-w_2)\setminus (L_1\cup L_2)$. Then by Lemma \ref{lem-leaf}, $(L_1,L_2,L_3)$ is a partition of $V(T)\setminus V(T_4)$ such that
\[
|L_1|\geq |L_0|\mbox{ and }|L_1|+|L_2|\geq |L_0|+p.
 \]
Let $X_1=V(T_2-w_2)$, $X_2=\{r_1,\ldots, r_p\}$ and $X_3=V(G')\setminus (V(T_4')\cup X_1\cup X_2)$. Then $(X_1,X_2,X_3)$ is a partition of $V(G')\setminus V(T_4')$ such that $|N_G(x)\cap X_1|\leq |L_0|\leq |L_1|$ for any $x\in X_3\cup\{h\}$ by \eqref{eq-T3H} and by the contradiction assumption. Now we verify the degree conditions in Corollary \ref{cor-tree}. In fact, for $x\in X_2\cup X_3\cup \{h\}$, $d_{G'}(x)\geq d_G(x)-|N_G(x)\cap (V(B)\cup \{w_2\})|\geq \delta -(\delta -m +1)=m-1$ since $x\notin X$, and for any $x\in X_3\cup \{h\}$, $|N_{G'}(x)\setminus X_1|= d_{G'}(x)-|N_G(x)\cap X_1|\geq m-1-|L_0|\geq m-1-|L_1|$, $|N_{G'}(x)\setminus (X_1\cup X_2)|\geq m-1-|L_0|-|X_2| = m-1-|L_0|-p \geq m-1-|L_1|-|L_2|$. 
Then by Corollary \ref{cor-tree}, $G'$ contains a tree $T'\cong T$. Then $w_2$ is a desired vertex in (c), a contradiction. The claim is proved.

Now we are ready to prove the result. In fact, by Claim 2, there exists $w_3r\in E(G)$ such that $w_3\in V(T_3)$ and $r\in V(R_i-r_i)$ for some $i$. Let $w_4$ be a child of $w_3$ in $T_2$. Furthermore, we may assume such $w_3r$ and $w_4$ satisfy that the subtree $T_{w_4}$ of $T$ is as small as possible. Without loss of generality, we may assume that $i=1$. By the choice of $w_3$,
\begin{equation}\label{eq-Tw4R}
N_G(r')\cap V(T_{w_4})\cap V(T_3)=\emptyset \mbox{ for any }r'\in \bigcup_{i=1}^p V(R_i-r_i).
\end{equation}
 Let $G'=G-(V(B)\cup \{w_4\})$. We consider the following two cases.

\indent\textbf{Case 1.} $V(T_{w_4})\cap V(R_1)\neq\emptyset$.

In this case, $V(R_1)\subseteq V(T_{w_4})$ since both them are rooted subtrees of $T$.
Let $T_5=T-V(T_{w_4}-w_4)$ and $T_5'=T-V(T_{w_4}-r)+w_3r$.
Then $T_5'\cong T_5\subseteq T$.

Define $L_1=Leaf(T_{w_4})$, $L_2=Leaf(T_{w_4}-L_1)$ and $L_3=V(T_{w_4}-w_4)\setminus (L_1\cup L_2)$. Then by Lemma \ref{lem-leaf}, $(L_1,L_2,L_3)$ is a partition of $V(T)\setminus V(T_5)$ such that
\[
|L_1|\geq |L_0\cap V(T_{w_4})|\mbox{ and }|L_1|+|L_2| \geq |L_0\cap V(T_{w_4})| + |\{r_1,\ldots, r_p\}\cap V(T_{w_4})|.
\]
Let $X_1=V(T_{w_4}-w_4)\cap V(T_2)$, $X_2=\{r_1,\ldots, r_p\}\cap V(T_{w_4})$ and $X_3=V(G')\setminus (V(T_5')\cup X_1\cup X_2)$. Then $(X_1,X_2,X_3)$ is a partition of $V(G')\setminus V(T_5')$ such that $|N_{G'}(x)\cap X_1|\leq |L_0\cap V(T_{w_4})|\leq |L_1|$ for any $x\in X_3\cup \{r\}$ by \eqref{eq-T3H} and \eqref{eq-Tw4R}. Now we verify the degree conditions in Corollary \ref{cor-tree}. In fact, noting that $\{v\mid N_{T}(v)\not\subseteq V(T_5)\}=\{w_4\}$, for any $x\in X_2\cup X_3\cup \{r\}$, $d_{G'}(x)\geq d_G(x)-|N_G(x)\cap (V(B)\cup \{w_4\})|\geq m-1$, and for any $x\in X_3\cup \{r\}$, $|N_{G'}(x)\setminus X_1|= d_{G'}(x)-|N_{G'}(x)\cap X_1|\geq m-1-|L_0\cap V(T_{w_4})|\geq m-1-|L_1|$, $|N_{G'}(x)\setminus (X_1\cup X_2)|\geq m-1-|L_0\cap V(T_{w_4})|-|X_2|
\geq m-1-|L_1|-|L_2|$. Then by Corollary \ref{cor-tree}, $G'$ contains a tree $T'\cong T$. However, $G[V(B)\cup \{w_4\}]$ is contained in a block of $G-V(T')$, a contradiction.

%
%

 \indent\textbf{Case 2.} $V(T_{w_4})\cap V(R_1)=\emptyset$.

 In this case, let $T_6=T[V(T)\setminus (V(T_{w_4}-w_4)\cup V(R_1-\{r_1\}\cup N_{R_1}(r_1)))]$ and $T_6'$ is obtained from $T_6$ by replacing $r$ with a neighbor of $r_1$ in $H$ if $r\in N_{R_1}(r_1)$,
and then replacing $w_4$ with $r$. Then $T_6'\cong T_6\subseteq T$.

Define $L_1=Leaf(T_{w_4})$, $L_2=Leaf(T_{w_4}-L_1)$ and $L_3=V(T)\setminus (V(T_6)\cup L_1\cup L_2)$. Then by Lemma \ref{lem-leaf}, $(L_1,L_2,L_3)$ is a partition of $V(T)\setminus V(T_6)$ such that
\[
|L_1|\geq |L_0\cap V(T_{w_4})|\mbox{ and }|L_1|+|L_2| \geq |L_0\cap V(T_{w_4})| + |\{r_2,\ldots, r_p\}\cap V(T_{w_4})|.
\]
Let $X_1=(V(T_{w_4}-w_4))\cap V(T_2)$, $X_2=\{r_2,\ldots, r_p\}\cap V(T_{w_4})$ and $X_3=V(G')\setminus (V(T_6')\cup X_1\cup X_2)$. Then $(X_1,X_2,X_3)$ is a partition of $V(G')\setminus V(T_6')$ such that $|N_{G'}(x)\cap X_1|\leq |L_0\cap V(T_{w_4})|\leq |L_1|$ for any $x\in V(H)\cup \bigcup_{i=1}^p V(R_i-r_i)$ by \eqref{eq-T3H} and \eqref{eq-Tw4R}. Now we verify the degree conditions in Corollary \ref{cor-tree}. In fact, letting $X_4=\{r\}\cup (N_{T_6'}(r_1)\cap (V(H)\cup V(R_1)))$, for $x\in X_2\cup X_3\cup X_4$, $d_{G'}(x)\geq d_G(x)-|N_G(x)\cap (V(B)\cup \{w_4\})|\geq m-1$, and for any $x\in X_3\cup X_4$, $|N_{G'}(x)\setminus X_1|\geq d_{G'}(x)-|L_0\cap V(T_{w_4})|\geq m-1-|L_1|$ and $|N_{G'}(x)\setminus (X_1\cup X_2)|\geq m-1-|L_0\cap V(T_{w_4})|- |X_2|\geq 
m-1-|L_1|-|L_2|$. Then by Corollary \ref{cor-tree}, $G'$ contains a tree $T'\cong T$. However, $G[V(B)\cup \{w_4\}]$ is contained in a block of $G-V(T')$, a contradiction, which completes the proof.
\end{proof}


By Lemma \ref{lem-main}, we may confirm Conjecture \ref{conj-mader} for $k=2$.

\begin{theorem}\label{thm-2-conn}
  For every tree $T_0$ of order $m$, each 2-connected graph $G$ with $\delta(G)\geq m+2$ contains a subtree $T\cong T_0$ such that $G-V(T)$ is 2-connected.
\end{theorem}
\begin{proof}
  Since $\delta(G)\geq m+2$, there exists a subtree $T\cong T_0$. Then we may assume $T$ is such a tree such that the maximum block $B$ of $G-V(T)$ has order as large as possible. Let $H=G-V(B\cup T)$. We are to prove $V(H)=\emptyset$. In fact, by contradiction, assume that $V(H)\neq\emptyset$.

  Then $|N_G(v)\cap V(B)|\leq 1$ for each $v\in V(H)$ by the assumption of $B$. Also, since $G$ is 2-connected and $B\cup H$ is not 2-connected, $N_G(H)\cap V(T)\neq\emptyset$.
  If there is a vertex $v\in V(H\cup T)$ such that $|N_G(v)\cap V(B)|\geq 2$ and $H\cup T-v$ contains a tree $T'\cong T_0$, then $T'$ is a subtree of $G$ such that $G[V(B)\cup \{v\}]$ is contained in a block of $G-V(T')$, a contradiction to the choice of $T$. So applying Lemma \ref{lem-main} with $\delta=m+1$, we see that for any $v\in V(T\cup H)$, $|N_G(v)\cap V(B)|\leq m+1-m=1$.

  Since $G$ is 2-connected, there is a shortest path $Q$ with both ends in $B$ and inner vertices not in $V(B)$ such that $|V(Q)|\geq3$. By the assumption of $Q$, $|N_G(v)\cap V(Q)|\leq 3$ for each $v\notin V(B\cup Q)$. Thus $|N_G(v)\cap V(B\cup Q)|\leq 4$. In fact, if the equation holds then a shorter path $Q'$ with ends in $B$ by using $v$ can be found easily, a contradiction. So $|N_G(v)\cap V(B\cup Q)|\leq 3$ and thus $\delta(G-V(B\cup Q))\geq m+2-3=m-1$. This implies there is a $T'\cong T_0$ in $G-V(B\cup Q)$ such that $B\cup Q$ is contained in a block of $T-V(T')$, a contradiction to our assumption. This completes the proof.
\end{proof}

\section{3-connected graphs}

In this case, we deal with $3$-connected graphs. The main method is similar to the one for 2-connected graphs. However, in the proof of 2-connected graphs the subgraph $B$ is always 2-connected, since every 2-connected graph has an ear-decomposition. This no longer works for 3-connected graphs.
Instead, we maintain a subdivision of some simple 3-connected graph in proof. In the end, by the assumption of the minimum degree, the subdivision is in fact a 3-connected graphs. 

Let $G$ be a subdivision of some simple 3-connected graph.
An \textit{ear} of $G$ is a maximal path $P$ whose each internal vertex has degree 2 in $G$. Then any two ears of $G$ has different pair of ends.
Write $t(G)=|\{v\mid d_G(v)\geq3\}|$. Then $G$ is 3-connected if and only if $t(G)=|G|$. Let $X, Y\subseteq V(G)$. An $(X,Y)$-path is path with one end in $X$, the other end in $Y$ and internal vertices not in $X\cup Y$. If $|X|=\{x\}$ then we use $(x, Y)$-path instead of $(\{x\}, Y)$-path.

\begin{theorem}\label{thm-main3}
  For every tree $T_0$ of order $m$, each 3-connected graph $G$ with $\delta(G)\geq m+3$ contains a subtree $T\cong T_0$ such that $G-V(T)$ is 3-connected.
\end{theorem}

\begin{proof}
  Since $\delta(G)\geq m+3$, there exists a subtree $T\cong T_0$.  Let $B$ be a induced subgraph of $G-V(T)$ isomorphic to a subdivision of some simple 3-connected graph. Such $B$ does exist. In fact, since $\delta(G-V(T))\geq m+3-|T|=3$ and thus $G-V(T)$ contains a subdivision of $K_4$. Then the minimum induced subgraph containing a subdivision of $K_4$ is a candidate of $B$. Furthermore, we may assume $T, B$ are such subgraphs such that
  \begin{equation}\label{eq-choseB}
   \mbox{ $t(B)$ is as large as possible and then $|B|$ is as small as possible.}
   \end{equation}

  Let $H=G-V(B\cup T)$. If $V(H)=\emptyset$ then $B=G-V(T)$ is a subdivision of some simple 3-connected graph. In fact, if there exists a vertex $v\in V(B)$ such that $d_B(v)=2$ then $d_G(v)=d_B(v)+|N_G(v)\cap V(T)|\leq 2+m<\delta(G)$, a contradiction. This implies $B$ is 3-connected and the result holds. So, by contradiction, we may assume that $V(H)\neq\emptyset$.
Let $X=\{v\in V(B)\mid d_B(v)\geq 3\}$ and $Y=V(B)\setminus X$. Then $t(B)=|X|$.

\textbf{Claim 1.} For any $u\in V(H\cup T)$, if $|N_G(u)\cap V(B)|\geq3$ and $H\cup T-u$ contains a subtree $T'\cong T_0$ then there exists an ear $Q$ of $B$ such that $N_G(u)\cap V(B)\subseteq V(Q)$.

In fact, if there exists such a vertex $u$ such that $|N_G(u)\cap V(B)|\geq3$ and no ear of $Q$ contains $N_G(u)\cap V(B)$ then, it is easy to find three $(u, X)$-paths intersecting with each other only at $u$. Thus $B_1=G[V(B)\cup \{u\}]$ is the subdivision of some simple 3-connected graph with $t(B_1)>t(B)$, a contradiction to \eqref{eq-choseB}. Claim 1 is proved.

   \textbf{Claim 2.} For any $u\in V(H\cup T)$, if $H\cup T-u$ contains a subtree $T'\cong T_0$ then $|N_G(u)\cap V(B)|\leq 3$.

   Suppose, to the contrary, that $u$ is such a vertex such that $|N_G(u)\cap V(B)|\geq 4$. Then by Claim 1,  there exists an ear $Q$ such that $N_G(u)\cap V(B)\subseteq V(Q)$.
    Denote $Q=u_1u_2\ldots u_q$. Let $a=\min\{i\mid u_i\in N_G(v)\}$ and $b=\max\{i\mid u_i\in N_G(v)\}$. If $b-a\geq 3$ then let $B'=G[V(B-\{v_{a+1},\ldots, v_{b-1}\})\cup \{u\}]$ and we see that $B'$ is a subdivision of some simple 3-connected graph such that $t(B')=t(B)$ and $|B'|-|B|=1-(b-a-1)<0$, a contradiction to the \eqref{eq-choseB}. So $b-a\leq 2$ and Claim 2 is proved.


\textbf{Claim 3.} For any $v\in V(H\cup T)$, $|N_G(v)\cap V(B)|\leq 3$.

  Suppose, to the contrary, that there exists $v\in V(H\cup T)$ such that $|N_G(v)\cap V(B)|\geq4$. By Claim 2, $v\in V(T)$. Apply Lemma \ref{lem-main} and by Claim 2, we see that $N_G(H)\cap V(T)=\emptyset$. Then again by Claim 2, $\delta(H)\geq m+3-3= m$. It follows that $H$ contains $T'\cong T_0$. Then by Claim 2, $|N_G(v)\cap V(B)|\leq 3$, a contradiction. The claim is proved.

\textbf{Claim 4.} $Y\neq\emptyset$.

Suppose, to the contrary, that $Y=\emptyset$. Then $B$ is 3-connected. Since $G$ is 3-connected, for any $u\notin V(B)$, there exist three $(u,B)$-paths $P_1,P_2,P_3$ intersecting with each other only at $u$. Furthermore, we may assume that $P_1, P_2, P_3$ are such paths such that $\sum_{i=1}^3|P_i|$ is as small as possible. Let $B'=G[V(B)\cup V(P_1\cup P_2\cup P_3)]$. Then $t(B')>t(B)$ and $B'$ is still the subdivision of some simple 3-connected graph by the minimality of $P_i$'s. Let $h$ be a vertex of $G-V(B')$ with minimum degree. If $\delta(G-V(B'))<m-1$ then $|N_G(h)\cap V(B')|\geq d_G(h)-\delta(G-V(B'))\geq5$.  For $i=1,2,3$, let $v_i\in N_G(h)\cap V(B')$ such that $v_i\notin \{v_1,\ldots, v_{i-1}\}$ and the distance from $v_i$ to $B$ in $B'-\{v_1,\ldots, v_{i-1}\}$ is as small as possible. Then there exist three $(h, V(B))$-paths in $G[V(B')\cup \{h\}]$ using edges $hv_1, hv_2, hv_3$ and avoiding two neighbors of $h$ in $\bigcup_{i=1}^3 P_i$, a contradiction to the choice of $P_i$'s. Thus $\delta(G-V(B'))\geq m-1$ and then there exists a subgraph $T'\cong T_0$ in $G-V(B')$ such that $t(B')>t(B)$, a contradiction to \eqref{eq-choseB}.
 Claim 3 is proved.

By Claim 3 and the fact $G$ is 3-connected, there exists a $(Y, B)$-path $P$ edge-disjoint with $B$ such that the two ends of $P$ lie in no ear of $B$.
Furthermore, we may assume $B$, $P$ are such subgraphs and paths such that \eqref{eq-choseB} holds and then $|P|$ is as small as possible. Let $B'=G[V(B)\cup V(P)]$ and we will show that $\delta(G-V(B'))\geq m-1$.

Suppose, to the contrary that $\delta(G-V(B'))\leq m-2$. Let $h$ be a vertex of $G-V(B')$ with minimum degree and then $|N_G(h)\cap V(B')|\geq 5$. Assume $P$ is a $(u,v)$-path such that $u\in Y$. Let $Q_1$ be an ear of $B$ such that $u\in V(Q_1)$. By Claim 3, $|N_G(h)|\cap V(P)|\geq |N_G(h)\cap V(B')|-3\geq2$. Let $v_1, v_2\in N_G(h)\cap V(P)$ closest to $u,v$ in $P$, respectively. By the minimality of $P$, $|N_G(h)\cap V(P)|\leq 3$ and thus $|N_G(h)\cap V(B)|\geq 2$.
If $N_G(h)\cap V(B)=V(Q_1)\setminus Y$ then let $B_1=G[V(B)\cup \{h\}\setminus (V(Q_1)\cap Y)]$ and $P_1=hv_2Pv$. Then $t(B_1)=t(B)$, $|B_1|\leq |B|$ and $|P_1|<|P|$, a contradiction to the choice of $B$ and $P$. So we may assume there exists $w\in N_G(h)\cap (V(B)\setminus (V(Q_1)\setminus Y))$. If $w\notin V(Q_1)$
then let $P_1=uPv_1hw$ and if $w\in V(Q_1)$ then let $P_1=whv_2Pv$. Then $P_1$ is a $(Y,B)$-path such that no ear contains the two end of $P$. By the choice of $P$, $|P_1|\geq |P|$. Thus $N_G(h)\cap (V(P)\setminus V(B))=\{v_1,v_2\}$ and then $|N_G(h)\cap V(B)|\geq3$. By Claim 1, there exists an ear $Q_2$ of $B$ such that $N_G(h)\cap V(B)\subseteq V(Q_2)$. By Claim 2, $|N_G(h)\cap V(B)|=3$. Let $v_3$ be the middle neighbor of $h$ on $Q_2$ and let $B_2=G[V(B-v_3)\cup \{h\}]$. Then $B_2$ is also a subdivision of some simple 3-connected graph and $t(B_2)=t(B)$ and $|B_2|=|B|$. If $Q_2\neq Q_1$ then let $P_2=hv_1Pu$ and if $Q_2=Q_1$ then let $P_2=hv_2Pv$. In either case, $P_2$ is a $(Y', B_2)$-path shorter than $P$ (where $Y'=Y\cup \{h\}\setminus \{v_3\}$), a contradiction. Hence, $\delta(G-V(B'))\geq m-1$.

It follows that there exists a subgraph $T'\cong T_0$ in $G-V(B')$ such that $t(B')>t(B)$, a contradiction to \eqref{eq-choseB}. The proof is completed.
\end{proof}

\end{document}